\documentclass[9pt,shortpaper,twoside,web]{ieeecolor}
\usepackage{amsmath,amssymb,amsfonts}
\usepackage{graphicx}

\usepackage{textcomp}
\usepackage{cite}
\usepackage{generic}
\def\BibTeX{{\rm B\kern-.05em{\sc i\kern-.025em b}\kern-.08em
    T\kern-.1667em\lower.7ex\hbox{E}\kern-.125emX}}
\newtheorem{theorem}{Theorem}
\newtheorem{definition}{Definition}
\newtheorem{lemma}{Lemma}

\newtheorem{remark}{Remark}

\newtheorem{proposition}{Proposition}
\newtheorem{assumption}{Assumption}
\newtheorem{corollary}{Corollary}

\newcommand{\Rn}{\mathbb{R}^n}
\newcommand{\ra}{\rightarrow}

\newcommand{\p}{\partial}

\DeclareMathOperator*{\tr}{tr}
\DeclareMathOperator*{\dvg}{div}
\DeclareMathOperator*{\grad}{grad}
\DeclareMathOperator*{\Hess}{Hess}
\DeclareMathOperator*{\vol}{vol}

\begin{document}
\title{Stability Analysis of Trajectories on 
Manifolds with Applications to Observer and Controller Design}

\author{Dongjun Wu$^{1}$, Bowen Yi$^{2}$ and Anders
	Rantzer$^1$
	\thanks{*This project has received funding from the European Research
		Council (ERC) under the European Union's Horizon 2020 research
		and innovation programme under grant agreement No 834142
	(ScalableControl).}
	\thanks{$^{1}$ D. Wu and A. Rantzer are with the Department of Automatic
		Control, Lund
	University, Box 118, SE-221 00 Lund, Sweden {\tt\small
	dongjun.wu@control.lth.se}, {\tt\small anders.rantzer@control.lth.se}.}
	\thanks{$^{2}$ B. Yi is with Robotics Institute, University of
	Technology Sydney, Sydney, NSW 2006, Australia {\tt \small
bowen.yi@uts.edu.au}.}
}
\maketitle
\begin{abstract}
This paper examines the local exponential stability (LES) of trajectories for
nonlinear systems on Riemannian manifolds. We present necessary and sufficient
conditions for LES of a trajectory on a Riemannian manifold by analyzing the
complete lift of the system along the given trajectory.
These conditions are coordinate-free which
reveal fundamental relationships between exponential stability and incremental
stability in a local sense. We then apply these results to design tracking
controllers and observers for Euler-Lagrangian systems on manifolds;
a notable advantage of our design is that it visibly reveals the
effect of curvature on system dynamics and hence suggests compensation
terms in the controller and observer. Additionally, we revisit some
well-known intrinsic observer problems using our proposed method, which largely
simplifies the analysis compared to existing results.

\end{abstract}

\section{Introduction} \label{sec:intro}

Many physical systems are naturally modelled on Riemannian manifolds. The most important example may refer to a class of mechanical systems with configuration spaces being Riemannian manifolds, rather than Euclidean spaces \cite{bullo2019geometric}. Another well known example appears in quantum systems \cite{d2007introduction}, in which systems state lives on Lie groups \cite{jurdjevic1972control}.

It is well known that local stability of equilibria for systems, whose state
space is on Riemannian manifolds, can be analyzed via linearization in local
coordinate---similar to the case in Euclidean space---known as the Lyapunov
indirect method. In many practically important control tasks, we are very
interested in the stability of \emph{a particular solution} $X(\cdot)$, the problem
of which widely arises in observer design, trajectory tracking
\cite{bullo1999tracking}, orbital stabilization \cite{yi2020orbital}, and
synchronization \cite{andrieu2016transverse}. In Euclidean space, these tasks,
or equivalently the stability of a solution $X(\cdot)$, may be solved by introducing
an error variable and then studying the error dynamics, which is usually a
nonlinear time-varying system. In particular, the local exponential
stability (LES) of $X(\cdot)$ for a given nonlinear system can be characterized by
the linearization of its error dynamics near the trajectory.

A similar problem arises in contraction and incremental stability analysis
\cite{lohmiller1998contraction,forni2013differential}, in which we are
interested in the attraction of any two trajectories to each other,
rather than a particular one $X(\cdot)$. The basic idea is to explore the stability
of a linearized dynamics, regarded as first-order approximation, to obtain the
incremental stability of the given system.
Indeed, studying the stability of a \emph{particular} solution via first-order
approximation has already been used, which, from the perspective of incremental
stability, is known as partial (or virtual) contraction
\cite{wang2005partial,forni2013differential}. As discussed above, some
excitation conditions of the given trajectory may be needed to continue
stability analysis. A successful application may be found in
\cite{bonnabel2014contraction} for the stability of extended Kalman filter
(EKF).

For the system evolving on Riemannian manifolds, however, the stability analysis
of a solution $X(\cdot)$ is much more challenging. The difficulty arises from two
aspects. On one hand, the ``error dynamics'' for such a case is more
involved---there are, indeed, no generally preferred definition of tracking (or
observation, synchronization) errors---the induced Riemannian distance on
manifolds can hardly be used to derive error dynamics directly. In
practice, one has to choose an error vector according to the structure of the
manifold, see \cite{bullo1999tracking,lageman2009gradient,mahony2008nonlinear}
for examples. On the other hand, the alternative method, via first-order
approximation (or partial contraction), is non trivial to be applied to
Riemmanian manifolds, since it is usually a daunting task to calculate the
differential dynamics on Riemannian manifolds, and also some complicated
calculations of parallel transport are involved. Overcoming these two major
challenges is the main motivation of the paper.

To address this, we provide in this paper an alternative way to study LES
of trajectories on Riemannian manifolds, namely, LES will be characterized by
the stability of the \emph{complete lift} of the system along the trajectory, in
this way removing the need of obtaining error dynamics. Complete lift, or
tangent lift, has been used to study various control problems, see for example
\cite{cortes2005characterization, van2015geometric,
bullo2019geometric, Wu2021, bullo2007reduction}. Among the listed
references, the most relevant work to ours are \cite{bullo2007reduction,
Wu2021}. In \cite{bullo2007reduction} the authors have remarked that
the complete lift can be seen as a linearization procedure. However, the
verification of stability of the complete lift system is challenging since
it is a system living in the tangent bundle and thus how to effectively use the
aforementioned characterization to guide controller and observer design is
an open question. We address this question in this paper.

The main contributions of the paper are three folds.

\begin{itemize}
	\item[-] Establish the relationship between LES of a solution
		to the stability of the complete lift along this solution on
		a Riemannian manifold, which can be seen as the Lyapunov indirect
		method on manifolds. Then show that LES of a
		solution is equivalent to local contraction near the solution
		$X(\cdot)$.

	\item[-] Propose an alternative approach for analysis of LES based on the
		characterization of complete lift system. This novel approach
		obviates the calculation of complete lift and hence facilitates the
		analysis of local exponential stability and contraction. We
		demonstrate the efficiency of the proposed methods by revisiting
		some well-known research problems.

	\item[-] Two main types of application problems are studied, namely,
		controller and observer design, especially for mechanical
		systems on manifolds. These results largely simplify the
		analysis in some existing works.
		In particular, the proposed method is quite efficient for
		analyzing a class of systems called Killing systems.
\end{itemize}

{\em Notation.} Throughout this paper we use rather standard notations from
Riemannian geometry \cite{carmo1992riemannian, petersen2006riemannian}. Denote
$M$ the Riemannian manifold of dimension $n$, $\langle \cdot, \cdot\rangle$ the
metric, $\nabla$ the Levi-Civita connection, $R(X,Y)Z$ the Riemannian curvature, $\pi:TM \ra M$ the natural
projection of the tangent bundle. We use $\nabla$ and $\grad(\cdot)$
interchangeably to
represent the gradient operator. Let $\Hess(\cdot)$ be the Hessian, $\exp(\cdot)$
the exponential map, $P_x^y: T_x M \ra T_y M$ the parallel transport from $T_x
M$ to $T_y M$, $d(x,y)$ the Riemannian distance between $x$ and $y$,
and $B_c(x)=\{\xi\in M | d(\xi,x)\le c \}$ the Riemannian ball. Let
$\phi^f(t;t_0,x_0)$ be the flow of the equation $\dot{x}=f(x,t)$; and we
sometime write $\phi(\cdot)$ when clear from the context. The notation $L_f Y$
stands for Lie derivative of $Y$ along $f$. 


\section{Local Exponential Stability on Riemannian Manifolds}
\label{sec:loc-exp}

\subsection{Theory: LES and Complete Lift} \label{subsec:LES et Clift}
Consider a system 
\begin{equation} \label{sys:NL-Rie}
	\dot{x} = f(x,t) 
\end{equation} 
with the system state $x$ on the Riemannian manifold $M$, and $X(\cdot)$ a
particular solution, {\it i.e.}, $\dot{X}(t)= f(X(t),t)$ from the intial
condition $X(t_0) =X_0 \in M$. We study the local exponential
stability of the solution $X(t)$. Some definitions are recalled below.

\begin{definition}\label{def:stab of traj}\rm
	The solution $X(\cdot)$ of the system (\ref{sys:NL-Rie}) is 
	\emph{locally exponentially stable} (LES) if there exist positive
	constants $c, K$ and $\lambda$, all independent of $t_0$, such that
	\[
		d(x(t), X(t)) \le K d(x(t_0), X(t_0)) e^{-\lambda
		(t-t_0)}, \; 
	\]
	for all $t \ge t_0 \ge 0$ and  $x(t_0)$ satisfying $d(x(t_0),
	X(t_0))<c$.
\end{definition}
\begin{remark} \label{rmk:1}\rm
For the case that $X(t)$ is a trivial solution at an equilibrium, {\em i.e.}, $X(t) \equiv X_0,~\forall t\ge t_0 $, Definition \ref{def:
	CLift} coincides with the standard definition of LES of an equilibrium.  We should also notice the peculiarity of this definition---it may happen that the union of LES solutions forms into a dense set.
	For example, every solution of $\dot{x}=Ax$ is LES when $A$ is Hurwitz.
\end{remark}

We recall the definition of complete lift of a vector field, see
\cite{yano1973tangent, crampin1986applicable} for more detailed discussions.

\begin{definition}[Complete Lift] \rm
	\label{def: CLift}Consider the time-varying vector field $f(x,t)$. Given
	a point $v\in TM$, let $\sigma(t,s)$ be the integral curve of
	$f$ with $\sigma(s,s)=\pi(v)$. Let $V(t)$ be the vector field along
	$\sigma$ obtained by Lie transport of $v$ by $f$. Then $(\sigma,V)$
	defines a curve in $TM$ through $v$. For every $t\geq s$, the
	\emph{complete lift} of $f$ into $TTM$ is defined at $v$ as
	the tangent vector to the curve $(\sigma,V)$ at $t=s$. We denote this
	vector field by $\tilde{f}(v,t)$ , for $v\in TM$.
\end{definition}

\begin{definition} \label{defsys:lift}\rm
	Given the system (\ref{sys:NL-Rie}), and a solution $X(t)$. Define the
	complete lift of the system (\ref{sys:NL-Rie}) along $X(t)$ as 
	\begin{equation} \label{eqdef: clift}
		\dot{v} = \tilde{f}(v, t), \; v(t) \in T_{X(t)}M
	\end{equation}where $\tilde{f}$ is the complete lift of $f$ as in
	Definition \ref{def: CLift}.
\end{definition}

The most important property of the complete lift system is {\em linearity} at a
fixed fibre. We refer the reader to \cite{Wu2021} for coordinate expression of (\ref{eqdef: clift}). From this definition, one can easily verify that the solution to (\ref{eqdef:
clift}), {\em i.e.}, $v(t)$ has the property that $\pi v(t) = X(t)$. Hence we say that
(\ref{eqdef: clift}) defines a dynamical system along the particular solution $X(t)$.  

The following simple characterization is the theoretical basis of this
paper. It can be viewed as an analogue of the Lyapunov indirect method on
Riemannian manifolds.
\begin{theorem}\label{thm:lin}\rm
Assume the system \eqref{sys:NL-Rie} is forward complete for $t\ge 0$. If the solution $X(t)$ is LES, then the complete
	lift of the system \eqref{sys:NL-Rie}
	along $X(\cdot)$ is exponentially stable. If the solution $X(\cdot)$ is bounded, the
	converse is also true.
\end{theorem}
\begin{proof}
($\Longrightarrow$)	Assume that the solution $X(t)$ is LES. Denote the
minimizing normalized (\emph{i.e.} with unit speed) geodesic joining $X(t_0)$ to
$x(t_0)$
as $\gamma:[0, \hat{s}]\rightarrow M$, with $\gamma(0) = X(t_0), \quad \gamma(\hat s)
= x(t_0)$ and
$0\leq\hat{s}=d(X(t_0),x(t_0))$. Let
$ v_0\in TM $ with $\pi(v_0)=X(t_0)$ and
	$v_0=\gamma ^{\prime}(0)$, and $v(t)$ be the solution to the complete
	lift system (\ref{eqdef: clift}). Then
	\begin{equation}
		\hat{s}\left\vert v(t)\right\vert =d\left(
			\exp_{X(t)}\left( 
		\hat{s}v(t)\right) ,X(t) \right) ,   \label{eq:1}
	\end{equation}
	where $\exp_{x}:TM\rightarrow M$ is the exponential map, by choosing $\hat{s}$ sufficiently small such that $\exp$ is defined. Using the
	metric property of $d$, we have
\begin{align}
	&d\left(  \exp_{X(t)}\left( \hat{s}v(t)\right), X(t) \right)   \notag \\
	 \leq &d\left( \exp_{X(t)}\left( \hat{s}v(t)\right)
		,x(t)\right) +d(x(t),X(t)) \\
	 \leq &d\left( \exp_{X(t)}\left( \hat{s}v(t)\right)
		,x(t)\right) +K\hat{s}e^{-\lambda(t-t_{0})},   \label{eq:2}
\end{align}
where the second inequality follows from Definition \ref{def:stab of traj}.
Fixing $t$ at the moment and invoking (\ref {eq:1}) and (\ref{eq:2}) we get
\begin{equation}
\begin{aligned}
\left\vert v(t)\right\vert &
\leq \kappa(\hat s) +Ke^{-\lambda(t-t_{0})} \\
\kappa(\hat s) &
:= \frac{d\left(
\exp_{X(t)}\left( \hat{s}v(t)\right)
,x(t)\right) }{\hat{s}}
.   \label{eq:3}
\end{aligned}
\end{equation}
Note that more precisely $\kappa$ is a function of both $t$ and $\hat{s}$. 
But omitting the $t$ argument does not affect the following analysis.

Now we need to show the term $\kappa(\hat s)$ is of order $O(\hat{s})$. Since
$x(t_0)=\gamma(\hat{s})$, this term can be rewritten as
\begin{equation*}
\kappa(s)=\frac{d\left( \exp_{X(t)}\left( sv(t)\right)
,\phi(t;t_{0},\gamma(s))\right) }{s}
\end{equation*}
where we have replaced $\hat{s}$ by $s$. To this end, we consider two
functions $ \alpha_{1}(s)  =\exp_{X(t)}\left( sv(t)\right)$,
$\alpha_{2}(s) =\phi(t;t_{0},\gamma(s))$. Similarly, we have omitted the $t$
argument which does not affect the proof.
We have $\alpha_{1}(0)=\alpha_{2}(0)=X(t)$ and $\alpha_{1}^{\prime}(0)=
\alpha_{2}^{\prime}(0)=v(t).$ Thus
\begin{equation*}
	\kappa(s)=\frac{1}{s}d(\alpha_{1}(s),\alpha_{2}(s))=O(s),
\end{equation*}
where we have used Lemma \ref{lem: dist} given in Appendix. Now letting $\hat{s}\rightarrow0$ in (
\ref{eq:3}) and noticing that the geodesic is unit speed, we have
\[
	|v(t)| \le K|v(t_0)|e^{-\lambda (t-t_0)}, 
\]for any $v(t_0) \in T_{X(t_0)} M$. 

($\Longleftarrow$) A consequence of Proposition \ref{prop:LES in W} (see Section
\ref{subsec:contra}):
If
the complete lift along $X(\cdot)$ is ES, then the proof of Proposition
\ref{prop:LES in W} shows that the system is contractive on a bounded set $B_c$
and thus the LES of $X(\cdot)$.
\end{proof}

\begin{remark}
	Theorem \ref{thm:lin} provides a characterization for LES of trajectories
	on manifolds via complete lift. Unfortunately, the original form of this
	theoretical result lacks practical utility for applications.
	The main reason is that the
	complete lift on manifolds is difficult to obtain, and quite often, its 
	calculation of it relies on local coordinates, which is in conflict with
	the purpose (coordinate-free design)
	of this paper. To circumvent this issue, we propose an
	alternative approach in Section \ref{subsec:use} based on Theorem
	\ref{thm:lin}, which will be much more efficient to use.
	But we must emphasize that Theorem \ref{thm:lin} plays the
	fundamental role for the rest of the paper. 

\end{remark}

From Theorem \ref{thm:lin}, we can derive the following interesting corollary
which says that there no unbounded LES solution exists for \emph{autonomous} systems.
\begin{corollary}\rm
	For a time invariant system $\dot{x} = f(x)$, a LES  
	solution $X(t)$  should always be bounded and non-periodic.
\end{corollary}
\begin{proof}
	The complete lift of $\dot{x} = f(x)$ is
	$
		\dot{v} = \frac{\p f}{\p x} v, \; v \in T_x \Rn.
	$ Clearly, $v=\dot{x}$ is a solution to the complete lift system. Then
	by Theorem \ref{thm:lin},
	$
		|\dot{X}(t)| \le k |\dot{X}(0)| e^{-\lambda t},
	$ hence $X(t)$ cannot be periodic. Further more, $
		|X(t)| \le |X(0)| + \int_0^t k |\dot{X}(0)| e^{-\lambda s} ds
		        = |X(0)| + \frac{k|\dot{X}(0)|}{\lambda} (1 -
		       e^{-\lambda t})
		        < |X(0)| + \frac{k|\dot{X}(0)|}{\lambda}.
		       $
\end{proof}

In \cite[Lemma 1]{giaccagli2020sufficient}, the authors obtain a similar result for autonomous systems, {\em i.e.}, there is a unique attractive equilibrium in an invariant set, in which the system is incrementally exponentially stable.

\subsection{Contraction and LES} \label{subsec:contra}
Contraction theory has become a powerful tool for analysis and design of control systems, see 
\cite{lohmiller1998contraction, forni2013differential, andrieu2016transverse, ruffer2013convergent,
manchester2017control, aminzare2014contraction}
and the references therein.
In Section \ref{subsec:LES et Clift}, we have studied 
LES of solutions to the system (\ref{sys:NL-Rie}). In this subsection, we will
show the close connection between the proposed result and contraction analysis
on manifolds \cite{simpson2014contraction,
Wu2021}. The reader may refer to \cite{ruffer2013convergent, angeli2002lyapunov} for the case on Euclidean space.

We say that the system (\ref{sys:NL-Rie}) is
contractive on a set $C$ if there exist positive constants $K, \lambda$,
independent of $t_0$ such that
\begin{equation}
	d(\phi(t;t_0, x_1), \phi(t;t_0, x_2)) \le K d(x_1, x_2) e^{-\lambda
	(t-t_0)}, 
\end{equation} for all $x_1, x_2 \in C, t\ge t_0 \ge 0$. 
For technical ease, we have slightly modified the definition of contraction
by allowing the set $C$ to be not forward invariant.   
Based on Theorem \ref{thm:lin}, we have the following proposition, which can be viewed as a bridge from LES to local contraction. 

\begin{proposition} \label{prop:LES in W}\rm
	A bounded solution $X(t)$ to the system (\ref{sys:NL-Rie}) is LES if and
	only if there exists a constant $c$ such that the system
	\eqref{sys:NL-Rie} is contractive on a bounded set $B_c$ whose interior
	contains $X(\cdot)$. 
\end{proposition}
\begin{proof}
	Assume that $X(t)$ is LES. Then the complete lift system along $X(t)$
	is exponential stable (ES) by Theorem \ref{thm:lin}. By converse Lyapunov theorem, there exists a $\mathcal{C}^1$
	function $V(t,v)$, quadratic in $v$ satisfying 
	\begin{equation}\label{eq:5}
		c_1|v|^2 \le V(t,v) \le c_2 |v|^2, \ \forall v \in T_{X(t)}M 
	\end{equation}
	and 
	\begin{equation}
		\dot{V}(t,v) = \frac{\p V}{\p t}(t,v) + L_{\tilde{f}} V(t,v) \le
		-c_3 |v|^2, \ \forall v \in T_{X(t)}M,  
	\end{equation}for all $t\ge t_0 \ge 0$ and three positive constants
	$c_1,c_2,c_3$.
	
	Due to the smoothness of $V$, we have
	\[
		|\dot{V}(t,P_{X(t)}^{x(t)}v) - \dot{V}(t,v)| \le c_4
		d_{TM}(P_{X(t)}^{x(t)}v, v) = c_4 d_M(x(t),X(t)).
	\]
	Thus 
	\begin{align*}
		\sup_{\tiny{\substack{|w|=1, \\  w\in T_{x(t)}M}} } 
		\dot{V}(t,w)
		&=
		\sup_{\tiny{\substack{|v|=1, \\ v\in T_{X(t)}M}}}
		\dot{V}(t,P_{X(t)}^{x(t)}v) \\
		&= 
		\sup_{\tiny{\substack{|v|=1, \\ v\in T_{X(t)}M}}}
		 \dot{V}(t,v) + 
		\dot{V}(t,P_{X(t)}^{x(t)}v) - \dot{V}(t,v)  \\
		& \le -c_3 + c_4 d(x(t),X(t)) < -c_5 <0,
	\end{align*} 
	 for $c$ small enough such that $d(x(t),X(t))$ will be small
	enough for all $t\ge t_0$ when $x(t_0) \in B_{c}({X(t_0)})$. Since
	$\dot{V}$ is quadratic in $v$ (due to the linearity of the complete lift
	system and that $V(t,v)$ is quadratic in $v$), this implies
	\[
		\dot{V}(t,v) \le -c_5 |v|^2, \ \forall v\in T_{x(t)}M, \ t\ge t_0 
	\]
	for all $x(t_0) \in B_c({X(t_0)})$. Then the
	system \eqref{sys:NL-Rie} is contractive on $B_c:= \bigcup_{t_0 \ge 0}
	B_c(X(t_0))$ which is bounded as is $X(\cdot)$ (use Theorem 2 \cite{Wu2021}). 
	The converse is obvious, and hence the proof is completed.
\end{proof}

The following corollary is a straightforward consequence.

\begin{corollary}\rm
	Assume that the system (\ref{sys:NL-Rie}) has an equilibrium point
	$x_\star \in M$. Then $x_\star$ is LES if and only if there exists an open
	neighborhood of $x_\star$ on which the system is contractive.
\end{corollary}

In \cite{forni2015differential}, the authors proved similar result to this corollary for \emph{autonomous} systems in Euclidean space. The paper \cite{ruffer2013convergent} focuses on asymptotic stability and asymptotic
contraction, also in Euclidean space.  
\subsection{A More Usable Form}\label{subsec:use}

As remarked earlier, Theorem \ref{thm:lin} is not suitable for practical
applications due to the difficulty of calculating the complete lift system. In this
subsection, we propose a more usable version of Theorem \ref{thm:lin} (still intrinsic)
which will make the analysis of LES a routine task.

For reasons that will be clear later, we rename the state $x$ in the system
\eqref{sys:NL-Rie} as $q$. Fig. \ref{fig:q} is drawn to illustrate our idea.
In Fig. \ref{fig:q}, the solid curve
represents a trajectory of the system system \eqref{sys:NL-Rie}, say
$q:\mathbb{R}_{\ge 0} \to M$, whose velocity vectors are drawn as black arrows,
denoted $\dot{q}$. The dashed curves are flows of the initial curve
$\gamma: s \mapsto \gamma(s) \in M$. The blue arrows emanating from 
the curve $q$ are the (transversal) velocities of the dashed curve, denoted as
$q'$, or in precise language, $q' = \frac{\partial q(s,t)}{\partial s} $ for the
parameterized curve $(s,t) \mapsto q(s,t)$. We call $q'$ a variation along
$q(\cdot)$.

\begin{figure}[ht]
	\begin{center}
		\includegraphics[scale = 0.55]{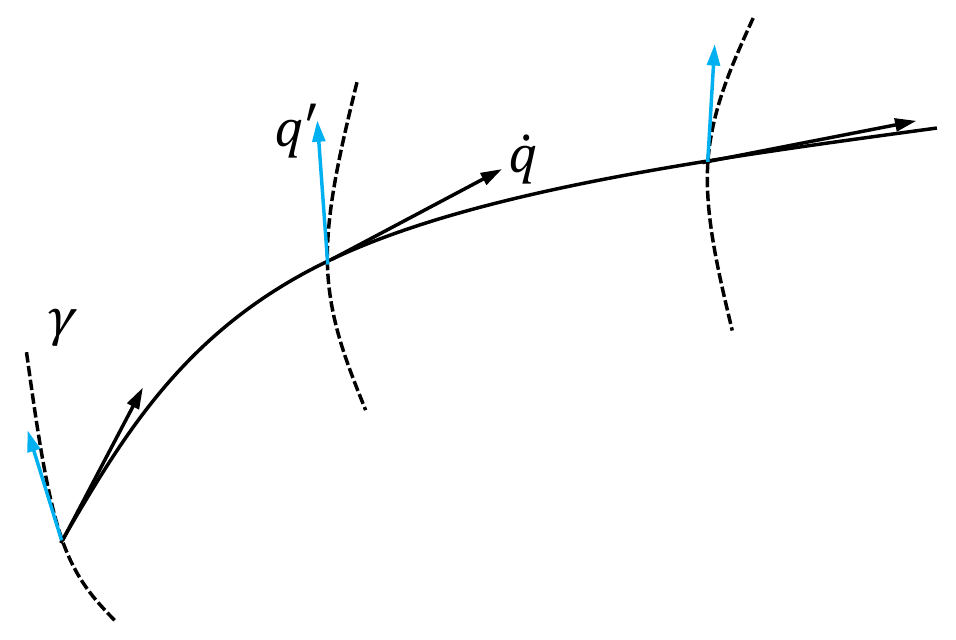}
		\caption{Illustration of $\dot{q}$ and $q'$.}
		\label{fig:q}
	\end{center}
\end{figure}

Two important observations can be made from the figure:
\begin{itemize}
	\item[-] By construction, $q'$ is the solution to the complete lift of the
		system \eqref{sys:NL-Rie} along the trajectory $q(\cdot)$.
		Thanks to this, the Lie bracket $[\dot{q},q']$ vanishes for all
		$t\ge t_0$ along $q(\cdot)$.
		\footnote{
		Recall that $[X,Y] = \left. \frac{d}{dt}\right|_{t_0}
			(\phi^X_t)^* Y(t_0)$, thus $[\dot{q},q'] = \left.
				\frac{d}{dt}\right|_{t_0} (\phi^f_t)^*
				(\phi_t^f)_* q'(t_0) = \left.  \frac{d}{dt}
					\right|_{t_0} q'(t_0) = 0$.
		} 
\item[-] The map $(s,t) \mapsto q(s,t)$ forms a parameterized surface in
		$M$. Then due to the torsion-free property of Levi-Civita
		connection, there holds $\frac{D }{dt}
		\frac{\partial q}{\partial s} = \frac{D}{ds} \frac{\partial
		q}{\partial t} $ (see \eqref{eq:swap-cov} and \cite[Lemma 3.4]{carmo1992riemannian}), 
		which implies that $\nabla_{\dot{q}} q' = \nabla_{q'}{\dot q} =
		\nabla_{q'} {f}$.

\end{itemize}

Now that $q'$ is the solution to the complete lift system, it is sufficient to
analyze the dynamics of $q'$. This may seem na\"ive at the first thought and
that the novelty seems to be only at a notational level. The fact is, however,
due to the above two observations, we now have access to rich results in
Riemannian geometry. In particular, we will see how LES on Riemannian
manifold is affected by curvature -- the most important ingredient of a
Riemannian manifold.


\subsection{Revisit of some existing results}\label{subsec:revisit}
\subsubsection{Contraction on Riemannian manifolds
\cite{simpson2014contraction}}
The following result is obtained in \cite{simpson2014contraction} (the
contraction version):
\begin{theorem}[\cite{simpson2014contraction}]\label{cor: cov}\rm
	Let $q(\cdot)$ be a solution to the system \eqref{sys:NL-Rie}, if
	\begin{equation*}
		\langle \nabla_{v} f, v \rangle \le -k  \langle
		v, v \rangle, \; \forall v \in T_{q(t)} M, \ t\ge 0,
	\end{equation*}for some positive constant $k$, then the solution $q(t)$
	is LES. 
\end{theorem}
The proof of this theorem will now simplify to a few lines:
\begin{proof}
	It suffices to show the exponential convergence of the metric $\langle q',
	q' \rangle$. Indeed,
	\begin{equation*}
		\frac{1}{2} \frac{d}{dt} \langle q', q' \rangle  
		= \left< \nabla_{\dot{q}} q',q' \right> 
		= \left< \nabla_{q'} f, q' \right> \le -k \left< q',q' \right>.
	\end{equation*} Thus $ \left<q',q'\right>$ converges exponentially.
\end{proof}

Notice that we have used the fact that $[q',\dot{q}]=0$.

\subsubsection{Intrinsic reduced observer \cite{bonnabel2010simple}}
The following lemma was among the key results in \cite{bonnabel2010simple}:
\begin{lemma}[\cite{bonnabel2010simple}] \label{lem:bonna}
	Let $M$ be a smooth Riemannian manifold. Let $P\in M$ be fixed. On the
	subspace of $M$ defined by the injectivity radius at $P$, we consider
	\begin{equation} \label{sys:bonna1}
		\dot{q} = -\frac{1}{2\lambda} \grad d(q,P)^2, \quad \lambda >0.
	\end{equation}
	If the sectional curvature is non-positive, the dynamics is a
	contraction in the sense of \cite{lohmiller1998contraction}, i.e., if
	$\delta x$ is a virtual displacement at fixed $t$, we have
	\begin{equation} \label{eq:bonna1}
		\frac{d}{dt} \left< \delta q, \delta q \right> \le -
		\frac{2}{\lambda} \left< \delta q, \delta q \right>.
	\end{equation} If the sectional curvature is upper bounded by $A>0$,
	then \eqref{eq:bonna} holds for $d(q,P)< \pi /(4 \sqrt{A})$.
\end{lemma}

The proof provided in \cite{bonnabel2010simple} is a bit technical. We now give a
simplified proof using the methods developed in this paper and provide a new
estimation of the convergence rate.

\begin{lemma} \label{lem:re-bonna}
	Let $M$ be a smooth Riemannian manifold whose curvature is upper bounded
	by $A\ge 0$. Let $P\in M$ be fixed. Then the dynamics \eqref{sys:bonna1}
	is globally contractive if $A=0$, and locally
	contractive otherwise, with contraction rate $ \gamma (q) = 
	\frac{2 \sqrt{A} d(q,P)}{\lambda \tan (\sqrt{A}d(q,P))}$,\footnote{$\gamma (P)$ is
		understood as $\lim_{d(q,P)\to 0} \gamma (q) =
	\frac{2}{\lambda}$. Notice that $\gamma$ is monotone decreasing and
	strictly positive on the interval $[0,\frac{\pi}{2})$} i.e.,
	\begin{equation} \label{eq:bonna}
		\frac{d}{dt} \left< \delta q, \delta q \right> \le -
		\gamma \left< \delta q, \delta q \right>.
	\end{equation} 
\end{lemma}
\begin{proof}
	Let $F(q) = \frac{1}{2} d(q,P)^2$ and we estimate
	\begin{align*}
		\frac{d}{dt} \left< q', q' \right> 
		& = 2 \left< - \frac{1}{\lambda} \nabla_{q'} \nabla F, q' \right> \\
		& = -\frac{2}{\lambda} \Hess F (q',q')
	\end{align*}where the last equality follows from the definition of
	the Hessian operator, see \eqref{eq:Hess}.
	The conclusion follows invoking comparison for the Hessian of square distance (e.g.,
	\cite[Theorem 6.6.1]{Jost2017}): 
	\begin{equation*}
		\Hess F \ge \sqrt{A} d(q,P) \cot (\sqrt{A} d(q, 
		 P)) \text{Id} 
	\end{equation*}for all $q\in \text{inj} (P)$ if $A > 0$ and for all $q
	\in M$ if $A=0$.
\end{proof}

\begin{remark}
The second part of the Lemma \ref{lem:bonna} \cite{bonnabel2010simple} seems incorrect: by
Rauch comparison (see \cite[Theorem 6.4.3]{petersen2006riemannian}), for
manifold with sectional curvature lower bounded by $k>0$, there holds
$\Hess F \le (1- k F ) g$, where $g$ is the
Riemannian metric. Therefore, the contraction rate is strictly less than
$\frac{2}{\lambda}$ in any neighborhood of $P$.
\end{remark}

\begin{remark}
	Since $\Hess F |_P = g$, if the Hessian is continuous at $P$, then 
	the dynamics \eqref{sys:bonna1} is always locally contractive without
	assumptions on curvature, which also implies that $P$ is an LES
	equilibrium.
\end{remark}

The above method is not limited to study contraction of distance, in fact, it
can be easily adapted to study $k$-contraction \cite{WCS2022} (Hausdorff measure such as area and
volume) on Riemannian manifolds. As an example, let us consider the contraction
of volume.
Suppose that $\{ q'_1 ,\cdots, q'_n \}$ forms a frame at $q$ and denote $ \vol
(q_1',\cdots, q'_n)$ the signed volume of the parallelepiped spanned by this
frame and we study the change of the volume under the dynamics \eqref{eq:bonna}:
\begin{equation}
\begin{aligned} 
	& \frac{d}{dt} \vol (q'_{1}, q'_{2}, \cdots, q'_{n}) \\
	= & -(\dvg \nabla F / \lambda) \vol (q'_{1}, q'_{2}, \cdots, q'_{n}) \\
	= & - \frac{\Delta F}{\lambda}  \vol (q'_{1}, q'_{2}, \cdots, q'_{n})
\end{aligned}
\end{equation}where $\Delta $ is the Laplace-Beltrami operator \cite{Jost2017}.
Now $\Delta F = \tr (G^{-1} \Hess F)$, with $G$ the Riemannian metric, we
can conclude that the condition in Lemma \ref{lem:re-bonna} implies exponential
contraction of volume (on
non-positive curvature manifold). Since $\Delta F$ is controlled by
$\tr (G^{-1} \Hess F)$, non-positive curvature assumption is too restrictive. 
In fact, $\Delta F  =  1 + H(q,P) d(q,P) $, where $H(q,P)$ is the mean
curvature, thus the same conclusion can be drawn for manifold with non-positive
mean curvature.

\begin{remark}
	From the proof of Lemma \ref{lem:re-bonna} we see that the function $F$
	need not be the square distance. It can be replaced by any function
	whose Hessian has the required property, as the next example shows.
\end{remark}

\subsubsection{Filtering on ${\it SO}(3)$}
Consider first the attitude control problem
\begin{equation} \label{sys:so(3)}
		\dot{R} = R u
	\end{equation} where $R \in SO(3)$ and the control input $u \in
	\mathfrak{so}(3)$. The control objective is to exponentially stabilize a
	solution $R_*(t) \in SO(3)$, which verifies
	$\dot{R_*}(t)=R_*(t)\Omega(t)$, where 
	$\Omega(t)$ is some known signal. 
	The Lie group $SO(3)$ is a Riemannian manifold
	with the bi-invariant metric $\langle X, Y \rangle = \tr (X^\top Y)$.
	Due to the bi-invariance of the metric, the Levi-Civita connection is
	simply $\nabla_X Y = \frac{1}{2} [X,Y]$, see \eqref{eq:Lie-Levi}.
	Consider the function 
	\[
		F(R,R_*) = \frac{1}{2} ||R - R_*||^2,
	\] where $||\cdot ||$ is the Frobenius norm ($F$ is not the square
	distance). The gradient and Hessian
	of $F$ can be calculated as $\nabla F = \frac{1}{2} R (R_*^\top R -
	R^\top R_*)$, $\Hess F (RY,RZ) = \frac{1}{4} \tr (Z^{\top} Y R^{\top}_*
	R )$, with $X,Y \in \mathfrak{so}(3)$ respectively. 
	Clearly, $R_*(\cdot)$ is the solution to 
	\begin{align*}
		\dot{R} & = -k \nabla F(R,R_*) + R\Omega(t) \\ 
			&= -\frac{k}{2} R(R_*^\top(t)R - R^\top R_*(t)) +
			R\Omega(t).
	\end{align*}
	Let us check the LES of the $R_*(\cdot)$. For $T_R SO(3) \ni R' = R X$
	for some $X \in \mathfrak{so}(3)$, we calculate 
	\begin{align*}
		\frac{1}{2} \frac{d}{dt} \left< R', R'\right> 
		= & -k \Hess F (R',R') + \left< \nabla_{R'} (R\Omega (t)), R'
		\right>  \\
		= & -k \Hess F (R',R') + \frac{1}{2} \left< [ R', R\Omega (t)], R' \right>
		\\
		= & -k \Hess F (R',R') + \frac{1}{2} \tr \{ (X^{\top} X - X X^{\top} ) \Omega \}
		\\
		= & -k \Hess F(R', R')
	\end{align*}since $X^{\top} X - X X^{\top}$ is symmetric.
	Note that the Hessian of $F$ is positive definite at $R=R_*$.
	Hence the controller 
	\[
		u = -\frac{k}{2} (R_*^\top(t)R - R^\top R_*(t)) + \Omega(t).
	\]renders the trajectory $R_*(\cdot)$ LES as expected.

	The extension to the design of a low pass filter becomes straightforward: the
	following dynamics
	\begin{equation}
		\dot{\hat{R}} = -\frac{k}{2} \hat{R}(R^\top\hat{R}-R^\top\hat{R}) +
		\hat{R}\Omega
	\end{equation} is a locally exponential observer (filter) for
	$\dot{R}=R\Omega$. This result has been obtained in
	\cite{lageman2009gradient}, see also \cite{mahony2008nonlinear}.

\subsection{Killing system} \label{subsubsec:Kill}
\subsubsection{Low-pass filter for Killing system}
Consider a system defined by a time-varying Killing field \cite[Chapter
8]{petersen2006riemannian} on a Riemannian manifold $(M,g)$:
\begin{equation} \label{sys:Kill}
	\dot{q}=f(t,q)
\end{equation}i.e., $L_f g  =0$, see also \eqref{eq:Killing} in Appendix. We
call such system a Killing system. When
the system \eqref{sys:Kill} is perturbed by some noise, it is tempting to design
a low pass filter to reconstruct the system state from the corrupted data $q$. For
that, we propose the simple filter
\begin{equation} \label{sys:filter-Kill}
	\dot{\hat{q}} = f(t,\hat{q}) - k \nabla F(\hat{q}, q),
\end{equation}where $F(q,p)=\frac{1}{2} d(q,p)^{2}$ and $k$ is a positive
constant. To verify the convergence of this filter, we calculate as before
\begin{align*}
	\frac{1}{2} \frac{d}{dt} \left< \hat{q}',\hat{q}' \right>  
	& = \left< \nabla_{\hat{q}'} (f-k\nabla F), \hat{q}' \right> \\
	& = \left< \nabla_{\hat{q}'} f, \hat{q}' \right> - k \left<
	\nabla_{\hat{q}'} \nabla F, \hat{q}'\right> \\
	& = - k \left< \nabla_{\hat{q}'} \nabla F, \hat{q}'\right>  \text{ (by
	\eqref{eq:Killing}) }\\
	& = - k \Hess F (\hat{q}', \hat{q}').
\end{align*} Since $\Hess$ is locally positive definite, the filter converges at
least locally. If in addition, the manifold has non-positive curvature, then the
convergence is global recalling that $\Hess F \ge g$ for manifold with
non-positive curvature. This is the case for the manifold of symmetric positive
definite matrices: ${\rm SPD} := \{ P \in \mathbb{R}^{n\times n}: P = P^{\top}
>0 \}$ equipped with the metric $ \left< X, Y \right> := \tr (X P^{-1} Y P^{-1})$
for $X,Y \in T_{P} {\rm SPD}$, see e.g., \cite{Criscitiello2022}.

\subsubsection{Discretization}
Let us now consider the discretization of this filter. First, notice that $f$ is
Killing, thus the discretization of \eqref{sys:Kill} may be written as $x_{k+1} =
\tau \cdot x_{k}$ for some isometric action $\tau \in \text{Iso}(M)$ ($\tau$ is
time-invariant since \eqref{sys:Kill} is autonomous). This model has been
considered in for example \cite{Tyagi2008}, where the manifold is the set of symmetric
positive definite matrices and it is called a ``linear system'' on Riemannian
manifolds.
Next, viewing $-k \nabla F
(\hat{q},q)$ as disturbance, we may discretize \eqref{sys:filter-Kill} as 
\[
	\hat{q}_{k+1} = \exp_{\tau \cdot \hat{q}_{k}} (k \Delta t
	\log_{\hat{q}_k} q_{k})
\] where $\Delta t$ is the sampling time and we use the standard notation
$\log_{x}y = \exp_x^{-1} y$, see e.g. \cite{Pennec2006}. 
	Indeed, fix $q$ and let $r(x) = d(x,q)$, then $ \nabla r(x) =
	-\frac{\exp_x^{-1} q}{||\exp_x^{-1} q||}$, see \cite[Lemma
	5.1.3]{Jost2017}. Thus $\nabla F = -\log_{\hat{q}} q $. For example, in
	Euclidean space $-\log_{\hat{q}}q = \hat{q}-q$, which is in accordance with
	$\nabla_{\hat{q}} \frac{1}{2} ||\hat{q} - q||^{2}$.
Let $e_{k} =
\log_{\hat{q}_k} q_k$ be the error, then 
\begin{align*}
	e_{k+1} & =\log_{\hat{q}_{k+1}} q_{k+1}  \\
	& = \log_{\exp_{\tau \cdot \hat{q}_{k}} (  k \Delta t \log_{\hat{q}_k}q_k)}
	\tau \cdot q_k \\
	& \approx \log_{\tau \cdot \hat{q}_k} \tau \cdot q_k  - k \Delta t
	\log_{\hat{q}_k}q_k \\
	& = D\tau \cdot \log_{\hat{q}_k} q_k - k \Delta t \log_{\hat{q}_k}q_k  \\
	& = (D \tau - k \Delta t \, {\rm Id}) e_k
\end{align*}By assumption, $D\tau$ is a linear isometric mapping (unitary),
therefore, $e_{k}$ tends to zero exponentially for all sufficiently small sampling time.

\subsubsection{Revisit of filter on ${\it SO}(3)$}
For a Lie group $G$ equipped with a left- (resp. right-) invariant metric $g$, it is known that
any right- (resp. left-) invariant vector fields are Killing fields, see for
example \cite[Chapter 8]{petersen2006riemannian}. Indeed, equip $G$ with a
right-invariant metric and consider 
a left-invariant vector field $V(x) := dL_{x} (v)$ for $v \in \mathfrak{g}$ and $x\in
G$, whose flow reads $F^{t}(x) = R_{\exp (tv)} x$; here $L$ and $ R$ represent
left and right action respectively. Thus $DF^{t} = dR_{\exp(tv)}$. Hence 
for any right-invariant vector fields
$W_1(x) = dR_x(w_1)  , W_2(x)=dR_x(w_2)$, for $w_1, w_2 \in \mathfrak{g}$, we
have 
\begin{align*}
	\left< DF^{t} (W_1), DF^{t} (W_2) \right> 
	& = \left< dR_{\exp (tv)} (W_1), dR_{\exp (tv)} (W_2) \right> \\
	& = \left< W_1, W_2 \right>_x \text{ (right-invariant metric)}\\
	& = \left<w_1, w_2\right>,
\end{align*}which is constant for all $t\ge 0$. Summarizing, a system
defined by left-invariant vector field (right-invariant is similar) e.g.,
\begin{equation}
	\dot{x} = dL_x (v(t)), \quad x \in G, \; v(t) \in \mathfrak{g}, \, \forall
	t\ge 0
\end{equation} is a Killing system when the underling metric of $G$ is
right-invariant. Therefore, a low-pass filter can be designed using
formula \eqref{sys:filter-Kill}. In particular, for the system \eqref{sys:so(3)} on
$SO(3)$, when $u$ is known, it reads $\dot{\hat{R}} = \hat{R}u + k
\log_{\hat{R}} R = \hat{
R} u + k \log {\hat{R}^{\top} R}$ when $SO(3)$ is equipped with the
standard bi-invariant metric. Thus we obtain another low-pass filter for
left-invariant dynamics on $SO(3)$.

\section{Application to Euler-Lagrangian Systems}\label{sec:Lag}
In this section, we utilize the proposed methods to study LES of trajectories of 
Euler-Lagrangian (EL) systems.
As pointed out in Section \ref{sec:intro}, trajectory tracking and observer design are
two typical important applications which involve the analysis of LES of
trajectories. Compared to the analysis of stability of equilibrium, these tasks
are generally much harder on manifolds. Most existing results
rely on calculations in local coordinates, which is usually a daunting task. We
demonstrate in this section that, the proposed approach can be
efficiently used to design and analyze controllers and observers for mechanical
systems, while obviating the calculation in local coordinates.


There are two pervasive approaches -- Lagrangian
and Hamiltonian -- for modelling of mechanical systems \cite{Abraham2008}. These two different
approaches have led to different design paradigms. Amongst the vast literature,
we mention two books that include some of the most important results in the two
fields: the book of R. Ortega {\it et. al.} \cite{Ortega2013} (Lagrangian
approach) and the book of van der Schaft
\cite{van2017l2} (Hamiltonian approach).

In this paper, we focus on the Lagrangian approach. Since we will work on
manifolds (the configuration space is a manifold rather than Euclidean), we
adopt the geometric modelling which is well documented in
\cite{bullo2019geometric}. Briefly speaking, one starts with a
configuration space $Q$ and then calculates the kinetic energy $\left< v_q, v_q
\right> $ and potential energy $V(q)$. The kinetic energy thus defines a
Riemannian metric on the configuration space, which depends only on the inertial of
system. 
Using principles of classical mechanics
(e.g., d'Alembert principle), one can derive the following so called
Euler-Lagrangian (EL) equation:
\begin{equation} \label{eq:EL}
	\nabla_{\dot{q}} \dot{q} = -\grad V(q) +\sum_{i=1}^{m} u_i B_{i}(q)
\end{equation}where $\dot{q}$ is the velocity, $\nabla$ is the Levi-Civita connection associated with the
metric, $B_i$ are some tangent vectors, and $u_i$ are external forces
(viewed as input in our setting) taking values in $\mathbb{R}$.

\subsection{Tracking controllers for EL systems}
Suppose that $(q_*(\cdot), \dot{q}_*(\cdot), u_*(\cdot))$ is a feasible pair of the system
\eqref{eq:EL}, i.e., $\nabla_{\dot{q}_*}\dot{q}_* = -\nabla V(q_*) +
\sum_{i=1}^{m} u_{i*}(t) B_{i}(q_*)$. The objective is to design a controller
$u(\cdot)$ such that $(q_*(\cdot), \dot{q}_*(\cdot))$ is LES. 
Before moving on however, we must stop for a moment to clarify
the statement that ``$(q_*(\cdot), \dot{q}_*(\cdot))$ is LES''. Unlike $q(\cdot)$, the curve
$(q_*(\cdot), \dot{q}_*(\cdot))$ lives in the tangent bundle $TM$, which is not
equipped with a distance function {\it a priori}. Thus in order to talk about convergence,
a topology should be defined on $TM$. This is achieved via
the so-called Sasaki metric \cite{yano1973tangent}. Due to the importance of
this metric, we briefly recall its construction in the following.

Let $V,W \in TTM$ be two tangent vectors at $(p,v)\in TM$ and 
\[
	\alpha: t \mapsto (p(t), v(t)), \, \beta: t \mapsto (q(t), w(t)),
\] are two curves in $TM$ with $p(0)= q(0) = p$, $v(0)= w(0) = v$, $\alpha'(0)=V$, $\beta'(0) = W$.
Define the inner product on $TM$ by 
\begin{equation}\label{sasaki}
	\left< V, W \right>_{\rm s} := 
	\left< p'(0), q'(0) \right> + \left< v'(0), w'(0) \right>
\end{equation} in which we write $v'(0) = \left. \frac{Dv(t)}{dt}\right|_{0}$.
The Sasaki metric is well-known to be a {\it bona fide} Riemannian metric on
$TM$. For details, see \cite{yano1973tangent}.

For a curve $w(s) =
(c(s), v(s))$ lying in $TM$, we can calculate its length under the Sasaki metric
as:
\begin{align*}
	\ell (w)  
	= & \int \sqrt{\langle w'(s), w'(s) \rangle_{\rm s}} ds \\
	= & \int \sqrt{\langle c'(s), c'(s) \rangle 
		+ \langle v'(s), v'(s) \rangle} ds 
\end{align*}in which $v'(s)$ is understood as the covariant derivative of
$v(\cdot)$ along $c(\cdot)$.

\begin{assumption}\rm
In the sequel we assume that for each pair of points $(q,v)$
and $(p,w)$ in $TM$, the minimizing geodesic that joins $(q,v)$ to $(p,w)$
always exists. 
\end{assumption}

Now that the EL equation \eqref{eq:EL} defines a system on $TM$, it
seems that to analyze LES of solutions of the EL equation, one has to consider
variation (see Section \ref{subsec:use}) of the form $(q',v')$, with $v' \in
TTM$. The next theorem shows that this is not needed.

\begin{theorem} \label{thm:lag-simp}
	Consider a dynamical system on a Riemannian manifold $(M,g)$:
	\begin{equation} \label{sys:nabla}
		\nabla_{\dot{q}} \dot{q} = f(q,\dot{q}) 
	\end{equation}
	where $f$ is smooth. Let $(q(\cdot), \dot{q}(\cdot))$ be a
	trajectory of the system and $q'$ any variation along $q(\cdot)$.
	Then the system \eqref{sys:nabla} is contractive if the following system
	\begin{equation} \label{sys:sasaki}
		\frac{D}{dt} 
		\begin{bmatrix} q' \\ \frac{Dq'}{dt} \end{bmatrix} 
		= F\left(
			\begin{bmatrix} q' \\ \frac{Dq'}{dt} \end{bmatrix} \right)
	\end{equation}is exponentially stable along any $q(\cdot)$.
\end{theorem}

\begin{remark}
Notice that $(q', Dq'/dt) \in T_q M \times T_q M $, thus exponential stability can be
defined in the obvious way for the system \eqref{sys:sasaki} using Sasaki metric.
\end{remark}

\begin{proof}
	Given a point $(q_1,v_1)\in TM$, and the integral
	curve $\eta_1(t) = ( q_1(t),
	\dot{q}_1(t))$ of the system \eqref{sys:nabla} passing through it at
	time $t=0$. Let $\eta_0(t) =
	(q(t), \dot{q}(t))$ be another integral curve with initial condition
	$(q_0, v_0)$. By
	assumption, there exists a minimizing geodesic
	$\gamma(s)=(q(s),v(s)), \ s\in[0,1]$ joining $(q_0, v_0)$ to
	$(q_1,v_1)$, that is, $\gamma(0)=(q_0,v_0),\ \gamma(1)=(q_1,v_1)$. Let
	$q(s,t)$ be the solution to the system \eqref{sys:nabla} with initial
	condition $\gamma(s)$, then the parameterized curve $s\mapsto
	(q(s,t),\frac{\p q(s,t)}{\p t})$ forms a
	variation between the curves $\eta_0(\cdot)$ and $\eta_1(\cdot)$. 
	Therefore, the following estimation of the distance between the two points
	$\eta_0(t)$ and $\eta_1(t)$ is obvious:
	\begin{equation}
		\begin{aligned}
		 d_{TM}(\eta_0(t), \eta_1(t)) 
		& \le  \int_0^1 \sqrt{\left|\frac{\p q}{\p s}(s,t)\right|^2 
			+ \left|\frac{D}{ds} \frac{\p q}{\p t} 
			\right|^2} ds \\
		& = \int_0^1 \sqrt{\left|\frac{\p q}{\p s}(s,t)\right|^2 
		+ \left|\frac{D}{dt} \frac{\p q}{\p s} \right|^{2} }ds
		\end{aligned}
	\end{equation} The conclusion follows immediately after replacing $\frac{\partial
	q}{\partial s}$ by $q'$.
\end{proof}

As we have remarked earlier, due to Theorem \ref{thm:lag-simp}, the analysis of
LES and contraction does not require variation of the form $(q',v')$ and that $q'$
alone is sufficient. This observation is crucial for the rest of this section.

With the preceding preparations, we are now in a position to study tracking
controller for the EL system. We focus on fully-actuated system:
\begin{equation} \label{eq:EL-full}
	\nabla_{\dot{q}} \dot{q} = -\grad V(q) +u 
\end{equation}
and assume $(q_*(\cdot),
\dot{q}_*(\cdot), u_*(\cdot) \equiv 0 )$ is a bounded feasible solution to the EL
equation, i.e., $\nabla_{\dot{q}_*} \dot{q}_{*} = - \nabla V(q_{*})$ (solution
with non-zero $u_*$ is similar). We propose
a controller with structure $u = u_{P} + u_{V} +u_{R}$ to locally
exponentially stabilize $(q_*(\cdot), \dot{q}_*(\cdot)$, where
\begin{equation} \label{ctrlu}
	\begin{aligned}
		& u_{P}(q) = - k_2 \nabla F(q,q_{*}), \\ 
		& u_{D}(q,\dot{q}) = - k_1 (\dot{q} - P_{q_*}^{q} \dot{q}_{*}), \\
		& u_{R}(q,\dot{q}_*) = R(\dot{q}, \nabla F(q,q_*)) \dot{q}
	\end{aligned}	
\end{equation}As before, $F$ is half of the square distance function.
$k_1$ and $k_2$ are constants to be determined and
$P_{q}^{p}$ is the parallel transport from $q$ to $p$, $R(\cdot,
\cdot)\cdot$ is the curvature tensor. Heuristically, this can be seen as a
PD-controller \cite{Ortega2013}, with a curvature compensation term.
By construction, $(q_*(\cdot),
\dot{q}_*(\cdot))$ is a solution to the closed loop system since $u(q_*,\dot{q}_*)
\equiv 0$. Hence it remains to show the LES of this solution. 

Thanks to Theorem \ref{thm:lag-simp} and Proposition \ref{prop:LES in W}, we
need only check the exponential stability of the system \eqref{sys:sasaki}
along $q_*(\cdot)$. For this we calculate
\begin{equation} \label{eq:nablaEL}
\begin{aligned}
	\nabla_{q'} \nabla_{\dot{q}}\dot{q} 
	& = \nabla_{\dot{q}} \nabla_{q'} \dot{q} + R(\dot{q}, q')\dot{q} \\
	& = \nabla_{\dot{q}} \nabla_{\dot{q}} q' + R(\dot{q}, q')\dot{q}  \\
	& = \frac{D^{2} q'}{dt^{2}} + R(\dot{q}, q')\dot{q}  
\end{aligned}
\end{equation} where we used the basic fact about the curvature tensor: $
\frac{D}{dt}\frac{D}{ds}X - \frac{D}{ds}\frac{D}{dt}X = R(\dot{q},q')X$, see
e.g., \cite[Lemma 4.1]{carmo1992riemannian}.
The following calculations are in order
(notice that we calculate along $q_*(\cdot)$, otherwise these are invalid):
\begin{equation} \label{eq:nabla u}
	\begin{aligned}
		 \nabla_{q'} u_{P} &=  - k_2 \nabla_{q'} \nabla F = k_2 q'   \\
		 \nabla_{q'} u_{D} &=  - k_1 \nabla_{q'} (\dot{q} - P_{q_*}^{q}
		\dot{q}_{*} ) = -  k_1 \nabla_{\dot{q}} q'\\
		 \nabla_{q'} u_{R} &= \nabla_{q'} R(\dot{q}, \nabla F) \dot{q}
		\\
				   & = (\nabla_{q'}R ) (\dot{q}, \nabla F)
				   \dot{q} + R(\nabla_{q'} \dot{q}, \nabla F)
				   \dot{q}  \\
				   & \quad + R(\dot{q}, \nabla_{q'} \nabla
				   F)\dot{q} +  R(\dot{q}, \nabla F)
				   \nabla_{q'}\dot{q} \\
				   & = R(\dot{q}, \nabla_{q'} \nabla F) \dot{q}
				   \\
				   & = R( \dot{q}, q') \dot{q}
	\end{aligned}
\end{equation} where we have used the fact that $\nabla_{q'}\nabla F(q,\dot{q}_*)
|_{q=q_*(t)} = q'$, $\nabla F(q_*,q_*) =0$. The second line of \eqref{eq:nabla
u} holds because one
can take $s \mapsto q(s,t)$ as a geodesic. Substituting \eqref{eq:nablaEL} and
\eqref{eq:nabla u} into the EL equation we immediately get
\begin{equation} \label{sys:linD}
	\frac{D^{2} q'}{dt} = - k_1 \frac{Dq'}{dt} - k_2 q' - \nabla_{q'} \nabla
	V.
\end{equation} 

\begin{theorem} \label{thm:track}
	Let $(q_*(\cdot), \dot{q}_*, u_* \equiv 0)$ be a bounded feasible solution to
	the fully-actuated Euler-Lagrangian system \eqref{eq:EL-full}.
	If the Hessian of the potential function $V$ is bounded along
	$q_*(\cdot)$, then the controller \eqref{ctrlu} renders $(q_*(\cdot),
	\dot{q}_*(\cdot)$ LES for $k_1>0$ and $k_2>0$ large enough.
\end{theorem}

\begin{proof}
If the Hessian of $V$ is bounded along $q_*(\cdot)$, then it is
obvious that the ``linear system'' \eqref{sys:linD} is exponentially stable
setting $k_1>0$ and
choosing $k_2>0$ large enough. The theorem follows invoking Theorem
\ref{thm:lag-simp}.
\end{proof}

\begin{remark}
	Note that the assumption of Theorem \ref{thm:track} holds if $V\in
	\mathcal{C}^{2}$ as $q_*(\cdot)$ is bounded.
	If $V$ is (weakly) convex, then the Hessian of $V$ is positive semi-definite,
	hence the same holds true for arbitrary positive constants $k_1, k_2$.
\end{remark}

\begin{remark}
In equation \eqref{sys:linD}, we have in fact obtained the celebrated Jacobi
equation by setting $u=0$ and $V=0$:
\begin{equation}\label{eq:Jacobi}
	\frac{D^2 q'}{dt^2} = -R(\dot{q},q')\dot{q}.
\end{equation}
Since in this case the EL equation reads $\nabla_{\dot{q}}\dot{q} =0$ (geodesic
equation), 
equation \eqref{eq:Jacobi} characterizes the effect of curvature to the geodesic flow.
The Jacobi equation plays a significant role in Riemannian geometry and has many
important implications. In order to help readers from the control community 
appreciate this equation, we now provide a control flavour to it.

For (\ref{eq:Jacobi}), choose a ``Lyapunov function''
\[
	V(\dot{q},q') = \langle \frac{D q'}{dt}, \frac{D q'}{dt} \rangle 
	+ \langle R(\dot{q}, q')\dot{q}, q' \rangle.
\]
Since we work only locally, let us consider a constant curvature manifold, that is
\[
	\langle R(\dot{q},q')\dot{q}, q'\rangle = K \langle \dot{q},\dot{q}
	\rangle \langle q', q' \rangle, \quad \forall \dot{q}, q'
\]for some constant $K$. The time derivative of $V$ reads
\begin{align*}
	\dot{V} &= 2 \langle \frac{D^2 q'}{dt^2},\frac{D q'}{dt} \rangle + 
	\langle R(\dot{q}, \frac{Dq'}{dt})\dot{q}, q'\rangle  + \langle
	R(\dot{q},q')\dot{q}, \frac{Dq'}{dt} \rangle \\
		&= 2 \langle -R(\dot{q},q')\dot{q},\frac{D q'}{dt} \rangle + 2
		\langle R(\dot{q},q')\dot{q}, \frac{Dq'}{dt} \rangle \\
		&= 0,
\end{align*} where we have used the fact that $\frac{D\dot{q}}{dt}=0$. Remember
that $q(\cdot)$ is a geodesic, we may assume $|\dot{q}|=1$, then it follows that
\[
	V(\dot{q},q') = |Dq'/dt|^2 + K |q'|^2  = \text{ constant}.
\] 
Therefore, we can draw the following non-rigorous conclusions:
\begin{itemize}
	\item $K>0$: along a given geodesic, nearby geodesics oscillate around
		it (see Fig. \ref{fig:K>0}).
	\item $K<0$: along a given geodesic, nearby geodesics have a trend to
		diverge.

	\item $K=0$: the geodesics neither converge nor diverge.
\end{itemize}
\end{remark}

\begin{figure}[ht]
\begin{center}
	\includegraphics[scale = 0.8]{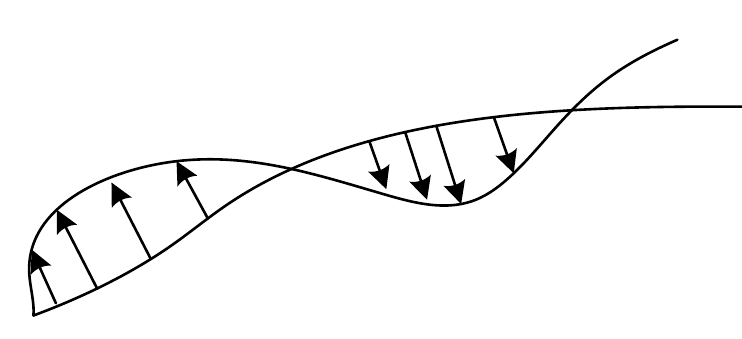}
	\caption{For $K>0$, the geodesics oscillate near a given geodesic.}
	\label{fig:K>0}
\end{center}
\end{figure}

In the above we have studied tracking controller design for fully-actuated EL
systems. This problem becomes more involved for under-actuated systems. In that
case, we may apply energy shaping method to obtain some matching conditions and
then try to solve some PDEs on the manifolds \cite{Ortega2002}, see also
\cite{Blankenstein2002} and the references therein.

\subsection{Speed Observer for EL Systems} \label{exmp:rouchon}
Consider the EL system without input
\begin{equation}\label{sys:lag-U}
	\nabla_{\dot{q}}\dot{q} =  -\nabla V(q)
\end{equation}where $V(q)$ is the potential energy. The objective is to design
a speed observer for $\dot{q}(\cdot)$ knowing $q(\cdot)$. 
In \cite{aghannan2003intrinsic}, Aghannan and Rouchon proposed the following
intrinsic speed observer for the system (\ref{sys:lag-U}) when there is no
potential energy in the EL equation:
\begin{equation} \label{sys:rouchon}
	\left\{
		\begin{aligned}
			\dot{\hat{q}} & = \hat{v} - \alpha \nabla F(\hat{q},q) \\
			\nabla_{\dot{\hat{q}}} \hat{v} &= - \beta \nabla
			F(\hat{q},q) + R(\hat{v}, \nabla F)\hat{v}.
		\end{aligned}
	\right.
\end{equation}where $F$ is half of the square distance as before. 
The convergence of this observer was analyzed in local coordinates via
contraction analysis \cite{aghannan2003intrinsic}, which was, in our opinion,
quite tedious.

\begin{remark}
	Using the notation introduced in Section \ref{subsubsec:Kill}, we may rewrite
	\eqref{sys:rouchon} as 
	\[
	\left\{
		\begin{aligned}
			\dot{\hat{q}} & = \hat{v} + \alpha \log_{\hat{q}} q\\
			\nabla_{\dot{\hat{q}}} \hat{v} &=  \beta \nabla
		\log_{\hat{q}} q - R(\hat{v}, \log_{\hat{q}}q)\hat{v}.
		\end{aligned}
	\right.
	\]obviating the use of the square distance function.
\end{remark}

In this subsection, we provide a much simpler proof using the methods developed
in this paper. Note that our model contains non-vanishing potential energy
function, thus it is an extension to the free Lagrangian case in
\cite{aghannan2003intrinsic}.


To cope with the potential energy, we consider a slightly modified
version of \eqref{sys:rouchon}:  
\begin{equation} \label{sys:rouchon+u}
	\left\{
		\begin{aligned}
			\dot{\hat{q}} & = \hat{v} - \alpha \nabla F(\hat{q},q) \\
			\nabla_{\dot{\hat{q}}} \hat v &= - \beta \nabla F(\hat{q},q) +
			R(\hat{v}, \nabla F)\hat{v} - P_{q}^{\hat{q}} \nabla
			V(q).
		\end{aligned}
	\right.
\end{equation}Note that by construction, $(q(\cdot), \dot{q}(\cdot))$ is a
solution to the observer. Hence it suffices to study LES of $(q(\cdot),
\dot{q}(\cdot))$.

Substituting $\hat{v} = \dot{\hat{q}} + \alpha \nabla F(\hat{q},q)$ into the
second line of \eqref{sys:rouchon+u}, we get
\begin{align*}
	\nabla_{\dot{\hat{q}}}(\dot{\hat{q}} + \alpha \nabla F)
	= & -\beta \nabla F + R(\dot{\hat{q}}+\alpha \nabla F, \nabla
	F)(\dot{\hat{q}}+\alpha \nabla F ) \\
	  & \; - P_{q}^{\hat{q}} \nabla V(q)
\end{align*}or
\begin{align*}
	\nabla_{\dot{\hat{q}}}\dot{\hat{q}} = - \alpha  \nabla_{\dot{\hat{q}}} \nabla F -
	& \beta \nabla F + R(\dot{\hat{q}}, \nabla F)(\dot{\hat{q}}+\alpha \nabla F) \\
	& - P_{q}^{\hat{q}} \nabla V(q)
\end{align*}
Taking covariant derivative along $q(\cdot)$ on both sides yields
\begin{equation}\label{eq:rouchon1}
	\nabla_{{q}'}\nabla_{\dot{\hat{q}}}\dot{\hat{q}} = \frac{D^2 {q}'}{dt^2} +
	R(\dot{\hat{q}},{q}')\dot{\hat{q}},
\end{equation} on the left, and
\begin{align*}	
	& -\alpha \nabla_{{q}'}\nabla_{\dot{\hat{q}}}\nabla F - \beta
	\nabla_{{q}'}
	\nabla F + \nabla_{{q}'}[R(\dot{\hat{q}}, \nabla F)(\dot{\hat{q}}+\alpha
	\nabla F)] \\
	 = & -\alpha \nabla_{\dot{\hat{q}}} \nabla_{{q}'} \nabla F - \alpha
	 R(\dot{\hat{q}},{q}')\nabla F - \beta \nabla_{{q}'} \nabla F  \\
	   & \; + \nabla_{{q}'}[R(\dot{\hat{q}}, \nabla F)(\dot{\hat{q}}+\alpha \nabla F)] \\
	 = & -\alpha \nabla_{\dot{\hat{q}}} {q}' - \beta  {q}' +
	 R(\dot{\hat{q}},\nabla_{{q}'}\nabla F)\dot{\hat{q}}   \\
	 = & -\alpha \nabla_{\dot{\hat{q}}}{q}' - \beta {q}' +
	 R(\dot{\hat{q}},{q}')\dot{\hat{q}},
\end{align*} on the right, where we have used the relations $\nabla F |_{\hat{q} = q} =0 $,
$\nabla_{q'} \nabla F |_{\hat{q}=q} = q'$ and $\nabla_{q'}P_{q}^{\hat{q}} \nabla
V(q) = 0$ (be $q'$ tangent to a geodesic).
Combining this with (\ref{eq:rouchon1}) yields
\begin{equation}
	 \frac{D^2q'}{dt^2} + \alpha\frac{Dq'}{dt}+\beta q' =0.
\end{equation} This, together with Theorem \ref{thm:lag-simp} shows the local
exponential convergence of the observer.

\begin{remark}
Notice that in both the tracking controller and observer design, we have to
calculate the geodesic distance. Although there are efficient computation
schemes, it is still tempting to avoid computing geodesics. This may be solved
by embedding the system into Euclidean space and use equivalent distance functions in
Euclidean spaces. The example of observer design on $SO(3)$ in Section
\ref{subsec:revisit} has used this method.
\end{remark}
\section{Conclusion}
In this paper, we have proposed a novel intrinsic approach for analyzing local
exponential stability of trajectories and contraction. The advantages of our
approach have been justified by applications and improved analysis of some existing
works in the literature. We leave studies of concrete examples including
under-actuated mechanical systems for future research.
\section{Acknowledgement}
We thank Prof. Antoine Chaillet, who gave important comments and suggestions
through the writing of the paper.

\section{Appendix}

We collect some elementary formulas in Riemannian geometry as a reference for the
reader. They can be found in standard texts such as \cite{carmo1992riemannian,
petersen2006riemannian}. Let $(M,g)$ be a smooth
Riemannian manifold. The Levi-Civita connection
on $M$ is compatible with the metric $g$: for any three vector fields $X,Y,Z \in
\Gamma(M)$, $
	X \left< Y, Z \right> = \left< \nabla_X Y, Z \right> + \left< Y,
	\nabla_X Z \right>$.
The Levi-Civita connection is torsion-free in the sense that $\nabla_{X} Y -
\nabla_Y X = [X,Y]$,
where $[X,Y]$ is the Lie bracket. Given a curve $q: t \mapsto
q(t)$ in $M$ and a vector field
$v(t)$ along $q(\cdot)$, the covariant derivative of $v(\cdot)$ along $q(\cdot)$
is defined as $\frac{Dv(t)}{dt} := \nabla_{\dot{q}(t)} v(t)$. Given a
$2$-surface parameterized by $(s,t) \mapsto q(s,t) $, then there holds
\begin{equation} \label{eq:swap-cov}
	\frac{D}{ds}\frac{\partial q}{\partial t}  = 
	\frac{D}{dt}\frac{\partial q}{\partial s}.
\end{equation}
The gradient of a scalar function $f$ on $M$ is defined as the unique vector
$\nabla f$ satisfying $
	\left< \nabla f, X \right> = df(X)$.
The Hessian of a scalar function is a symmetric bilinear form on
$TM$ defined as
\begin{equation} \label{eq:Hess}
	\Hess f (X,Y) := \left< \nabla_X \nabla f , Y \right>, \; \forall X,Y \in
	\Gamma(M).
\end{equation}
For a parameterized surface $(s,t)\to
q(s,t)$ and a vector field along the surface, there holds
\begin{equation}	
\frac{D}{ds}\frac{DX}{dt} - \frac{D}{dt}\frac{DX}{ds}  = R
	\left( \frac{\partial q}{\partial t}, \frac{\partial q}{\partial s} 
	\right)X.
\end{equation}

A metric on a Lie group $G$ is bi-invariant if it is both left-invariant,
i.e., $dL_{x} \left< v,w \right> = \left< v, w \right>$ and right-invariant. For
a bi-invariant metric, the Levi-Civita connection admits a simple formula
\begin{equation} \label{eq:Lie-Levi}
	\nabla_X Y = \frac{1}{2} [X,Y].
\end{equation}
A vector field $X$ on is called a Killing field (w.r.t. $g$) if $L_X g =0$.
Consequently, if $X$ is Killing, $Y$ an arbitrary vector field, there holds
\begin{equation} \label{eq:Killing}
	g(\nabla_Y X, Y)=0.
\end{equation}

\begin{lemma}\rm
\label{lem: dist}Given $\gamma_{1},\gamma_{2}\in\mathcal{C}^{1}(\mathbb{R}
_{+};M)$, where $M$ is a Riemannian manifold. If 
$\gamma_{1}(0)=\gamma_{2}(0)=x$
and $\gamma_{1}^{\prime}(0)=\gamma_{2}^{\prime}(0)=v$,
then $d(\gamma_{1}(s),\gamma_{2}(s))=O(s^{2})$
when $s>0$ is sufficiently small.
\end{lemma}

\bibliographystyle{IEEEtrans}
\bibliography{GeometricControl,IEEEabrv}

\begin{thebibliography}{10}
\providecommand{\url}[1]{#1}
\csname url@samestyle\endcsname
\providecommand{\newblock}{\relax}
\providecommand{\bibinfo}[2]{#2}
\providecommand{\BIBentrySTDinterwordspacing}{\spaceskip=0pt\relax}
\providecommand{\BIBentryALTinterwordstretchfactor}{4}
\providecommand{\BIBentryALTinterwordspacing}{\spaceskip=\fontdimen2\font plus
\BIBentryALTinterwordstretchfactor\fontdimen3\font minus
  \fontdimen4\font\relax}
\providecommand{\BIBforeignlanguage}[2]{{%
\expandafter\ifx\csname l@#1\endcsname\relax
\typeout{** WARNING: IEEEtran.bst: No hyphenation pattern has been}%
\typeout{** loaded for the language `#1'. Using the pattern for}%
\typeout{** the default language instead.}%
\else
\language=\csname l@#1\endcsname
\fi
#2}}
\providecommand{\BIBdecl}{\relax}
\BIBdecl

\bibitem{bullo2019geometric}
F.~Bullo and A.~D. Lewis, \emph{Geometric Control of Mechanical Systems:
  Modeling, Analysis, and Design for Simple Mechanical Control Systems}.\hskip
  1em plus 0.5em minus 0.4em\relax Springer, 2019, vol.~49.

\bibitem{d2007introduction}
D.~d'Alessandro, \emph{Introduction to Quantum Control and Dynamics}.\hskip 1em
  plus 0.5em minus 0.4em\relax CRC press, 2007.

\bibitem{jurdjevic1972control}
V.~Jurdjevic and H.~J. Sussmann, ``Control systems on {L}ie groups,''
  \emph{Journal of Differential Equations}, vol.~12, no.~2, pp. 313--329, 1972.

\bibitem{bullo1999tracking}
F.~Bullo and R.~M. Murray, ``Tracking for fully actuated mechanical systems: a
  geometric framework,'' \emph{Automatica}, vol.~35, no.~1, pp. 17--34, 1999.

\bibitem{yi2020orbital}
B.~Yi, R.~Ortega, D.~Wu, and W.~Zhang, ``Orbital stabilization of nonlinear
  systems via {M}exican sombrero energy shaping and pumping-and-damping
  injection,'' \emph{Automatica}, vol. 112, 2020, art. no. 108861.

\bibitem{andrieu2016transverse}
V.~Andrieu, B.~Jayawardhana, and L.~Praly, ``Transverse exponential stability
  and applications,'' \emph{IEEE Transactions on Automatic Control}, vol.~61,
  no.~11, pp. 3396--3411, 2016.

\bibitem{lohmiller1998contraction}
W.~Lohmiller and J.-J.~E. Slotine, ``On contraction analysis for non-linear
  systems,'' \emph{Automatica}, vol.~34, no.~6, pp. 683--696, 1998.

\bibitem{forni2013differential}
F.~Forni and R.~Sepulchre, ``A differential {L}yapunov framework for
  contraction analysis,'' \emph{IEEE Transactions on Automatic Control},
  vol.~59, no.~3, pp. 614--628, 2014.

\bibitem{wang2005partial}
W.~Wang and J.-J.~E. Slotine, ``On partial contraction analysis for coupled
  nonlinear oscillators,'' \emph{Biological Cybernetics}, vol.~92, no.~1, pp.
  38--53, 2005.

\bibitem{bonnabel2014contraction}
S.~Bonnabel and J.-J. Slotine, ``A contraction theory-based analysis of the
  stability of the deterministic extended {K}alman filter,'' \emph{IEEE
  Transactions on Automatic Control}, vol.~60, no.~2, pp. 565--569, 2014.

\bibitem{lageman2009gradient}
C.~Lageman, J.~Trumpf, and R.~Mahony, ``Gradient-like observers for invariant
  dynamics on a {L}ie group,'' \emph{IEEE Transactions on Automatic Control},
  vol.~55, no.~2, pp. 367--377, 2009.

\bibitem{mahony2008nonlinear}
R.~Mahony, T.~Hamel, and J.-M. Pflimlin, ``Nonlinear complementary filters on
  the special orthogonal group,'' \emph{IEEE Transactions on Automatic
  Control}, vol.~53, no.~5, pp. 1203--1218, 2008.

\bibitem{cortes2005characterization}
J.~Cort{\'e}s, A.~van~der Schaft, and P.~E. Crouch, ``Characterization of
  gradient control systems,'' \emph{SIAM Journal on Control and Optimization},
  vol.~44, no.~4, pp. 1192--1214, 2005.

\bibitem{van2015geometric}
A.~van~der Schaft, ``A geometric approach to differential {H}amiltonian systems
  and differential {R}iccati equations,'' in \emph{2015 54th IEEE Conference on
  Decision and Control (CDC)}.\hskip 1em plus 0.5em minus 0.4em\relax IEEE,
  2015, pp. 7151--7156.

\bibitem{Wu2021}
D.~Wu and G.-R. Duan, ``Further geometric and {L}yapunov characterizations of
  incrementally stable systems on {F}insler manifolds,'' \emph{IEEE
  Transactions on Automatic Control}, vol.~67, no.~10, pp. 5614--5621, 2021.

\bibitem{bullo2007reduction}
F.~Bullo and A.~D. Lewis, ``Reduction, linearization, and stability of relative
  equilibria for mechanical systems on {R}iemannian manifolds,'' \emph{Acta
  Applicandae Mathematicae}, vol.~99, no.~1, pp. 53--95, 2007.

\bibitem{carmo1992riemannian}
M.~P.~d. Carmo, \emph{Riemannian Geometry}.\hskip 1em plus 0.5em minus
  0.4em\relax Birkh{\"a}user, 1992.

\bibitem{petersen2006riemannian}
P.~Petersen, S.~Axler, and K.~Ribet, \emph{Riemannian Geometry}.\hskip 1em plus
  0.5em minus 0.4em\relax Springer, 2006, vol. 171.

\bibitem{yano1973tangent}
K.~Yano and S.~Ishihara, \emph{Tangent and {C}otangent {B}undles:
  {D}ifferential {G}eometry}.\hskip 1em plus 0.5em minus 0.4em\relax Dekker,
  1973, vol.~16.

\bibitem{crampin1986applicable}
M.~Crampin and F.~Pirani, \emph{Applicable Differential Geometry}.\hskip 1em
  plus 0.5em minus 0.4em\relax Cambridge University Press, 1986, vol.~59.

\bibitem{giaccagli2020sufficient}
M.~Giaccagli, D.~Astolfi, V.~Andrieu, and L.~Marconi, ``Sufficient conditions
  for output reference tracking for nonlinear systems: A contractive
  approach,'' in \emph{59th IEEE Conference on Decision and Control}, 2020.

\bibitem{ruffer2013convergent}
B.~S. R{\"u}ffer, N.~Van De~Wouw, and M.~Mueller, ``Convergent systems vs.
  incremental stability,'' \emph{Systems \& Control Letters}, vol.~62, no.~3,
  pp. 277--285, 2013.

\bibitem{manchester2017control}
I.~R. Manchester and J.-J.~E. Slotine, ``Control contraction metrics: Convex
  and intrinsic criteria for nonlinear feedback design,'' \emph{IEEE
  Transactions on Automatic Control}, vol.~62, no.~6, pp. 3046--3053, 2017.

\bibitem{aminzare2014contraction}
Z.~Aminzare and E.~D. Sontag, ``Contraction methods for nonlinear systems: A
  brief introduction and some open problems,'' in \emph{53rd IEEE Conference on
  Decision and Control}.\hskip 1em plus 0.5em minus 0.4em\relax IEEE, 2014, pp.
  3835--3847.

\bibitem{simpson2014contraction}
J.~W. Simpson-Porco and F.~Bullo, ``Contraction theory on {R}iemannian
  manifolds,'' \emph{Systems \& Control Letters}, vol.~65, pp. 74--80, 2014.

\bibitem{angeli2002lyapunov}
D.~Angeli, ``A {L}yapunov approach to incremental stability properties,''
  \emph{IEEE Transactions on Automatic Control}, vol.~47, no.~3, pp. 410--421,
  2002.

\bibitem{forni2015differential}
F.~Forni, A.~Mauroy, and R.~Sepulchre, ``Differential positivity characterizes
  one-dimensional normally hyperbolic attractors,'' \emph{arXiv preprint
  arXiv:1511.06996}, 2015.

\bibitem{bonnabel2010simple}
S.~Bonnabel, ``A simple intrinsic reduced-observer for geodesic flow,''
  \emph{IEEE Transactions on Automatic Control}, vol.~55, no.~9, pp.
  2186--2191, 2010.

\bibitem{Jost2017}
J.~Jost, \emph{Riemannian Geometry and Geometric Snalysis, {S}eventh
  {E}dition}.\hskip 1em plus 0.5em minus 0.4em\relax Springer, 2017.

\bibitem{WCS2022}
C.~Wu, I.~Kanevskiy, and M.~Margaliot, ``$k$-contraction: Theory and
  applications,'' \emph{Automatica}, vol. 136, p. 110048, 2022.

\bibitem{Criscitiello2022}
C.~Criscitiello and N.~Boumal, ``An accelerated first-order method for
  non-convex optimization on manifolds,'' \emph{Foundations of Computational
  Mathematics}, pp. 1--77, 2022.

\bibitem{Tyagi2008}
A.~Tyagi and J.~W. Davis, ``A recursive filter for linear systems on
  {R}iemannian manifolds,'' in \emph{2008 IEEE Conference on Computer Vision
  and Pattern Recognition}.\hskip 1em plus 0.5em minus 0.4em\relax IEEE, 2008,
  pp. 1--8.

\bibitem{Pennec2006}
X.~Pennec, P.~Fillard, and N.~Ayache, ``A {R}iemannian framework for tensor
  computing,'' \emph{International Journal of Computer Vision}, vol.~66, pp.
  41--66, 2006.

\bibitem{Abraham2008}
R.~Abraham and J.~E. Marsden, \emph{Foundations of Mechanics}.\hskip 1em plus
  0.5em minus 0.4em\relax American Mathematical Soc., 2008, no. 364.

\bibitem{Ortega2013}
R.~Ortega, A.~Lor{\'i}a, P.~J. Nicklasson, and H.~Sira-Ram{\'i}rez,
  \emph{Passivity-based Control of {E}uler-{L}agrange Systems: Mechanical,
  Electrical and Electromechanical Applications}.\hskip 1em plus 0.5em minus
  0.4em\relax Springer Science \& Business Media, 2013.

\bibitem{van2017l2}
A.~van~der Schaft, \emph{$L_2$-gain and Gassivity Techniques in Nonlinear
  Control}.\hskip 1em plus 0.5em minus 0.4em\relax Springer, 2017.

\bibitem{Ortega2002}
R.~Ortega, M.~W. Spong, F.~G{\'o}mez-Estern, and G.~Blankenstein,
  ``Stabilization of a class of underactuated mechanical systems via
  interconnection and damping assignment,'' \emph{IEEE Transactions on
  Automatic Control}, vol.~47, no.~8, pp. 1218--1233, 2002.

\bibitem{Blankenstein2002}
G.~Blankenstein, R.~Ortega, and A.~J. Van Der~Schaft, ``The matching conditions
  of controlled {L}agrangians and {IDA}-passivity based control,''
  \emph{International Journal of Control}, vol.~75, no.~9, pp. 645--665, 2002.

\bibitem{aghannan2003intrinsic}
N.~Aghannan and P.~Rouchon, ``An intrinsic observer for a class of {L}agrangian
  systems,'' \emph{IEEE Transactions on Automatic Control}, vol.~48, no.~6, pp.
  936--945, 2003.

\end{thebibliography}

\end{document}